\documentclass[11pt,a4paper]{article}

\usepackage{amssymb}
\usepackage{tikz}
\usepackage[textures]{epsfig}
\usepackage{amscd}
\usepackage{mathrsfs}
\usepackage{xcolor}
\usepackage{cancel}
\usepackage[centertags]{amsmath}
\usepackage{amsfonts}
\usepackage{amsthm}
\usepackage{epsfig}
\usepackage{graphicx}
\usepackage[T1]{fontenc}
\usepackage{latexsym}
\usepackage{amsfonts}
\usepackage{fancyhdr}
\usepackage{mathtools}
\mathtoolsset{showonlyrefs}

\def\ni{\noindent}
\def\beq{\arraycolsep1pt\begin{eqnarray*}}
\def\eeq{\end{eqnarray*}}

\newcommand{\Z}{{\mathbb Z}}

\newcommand{\R}{{\mathbb R}}

\newcommand{\half}{{\textstyle{\frac{1}{2}}}}

\newcommand\BB{\mathcal{B}}

\newcommand\ov{\overline}

\newtheorem{Thm}{Theorem}[section]
\newtheorem{lem}[Thm]{Lemma}
\newtheorem{pro}[Thm]{Proposition}

\def\bea{\begin{eqnarray}}
\def\eea{\end{eqnarray}}
\numberwithin{equation}{section}

\date{}

\title{A multiplicity result for Hamiltonian systems with mixed periodic-type and Neumann-type\\ boundary conditions }

\author{Wahid Ullah}

\begin{document}

\maketitle

\begin{abstract}
We investigate the multiplicity of solutions for a Hamiltonian system coupling two systems associated with mixed boundary conditions. Corresponding to the first system, we impose periodic boundary conditions and assume the twist assumption commonly used for the Poincar\'e--Birkhoff theorem, while for the second one, we consider a two-point boundary conditions of Neumann type.
\end{abstract}

\section{Introduction and statement of the main result}

The classical Poincar\'e--Birkhoff fixed point theorem, also known as Poincar\'e’s last geometric theorem, states that any area-preserving homeomorphism of a planar annulus that keeps both boundary circles invariant and twisting them in opposite directions must have at least two fixed points. Poincar\'e~\cite{poincare:1912} conjectured this theorem shortly before his death in 1912, and it was later proved by Birkhoff \cite{Bir1913TAMS, Bir1925AM}. Poincar\'e and Birkhoff were motivated by its applications to the search of periodic solutions of conservative dynamical systems.

\medbreak

Arnold sought to extend the Poincar\'e--Birkhoff Theorem to higher dimensions, emphasizing its significance for understanding periodic solutions in systems with many degrees of freedom~\cite[page 416]{Arn1978SB}. Birkhoff himself had already recognized the importance of this issue, describing it as an outstanding question~\cite[page 299]{Bir1925AM}. He \cite{Bir1931Paris} proposed a $2N$-dimensional version of the theorem, which involved the key assumption that there exists a manifold diffeomorphic to the $N$-torus, where the exact symplectic map preserves the first $N$ coordinates.

\medbreak
In 1978, Rabinowitz~\cite{Rab1978CPAM} demonstrated how the periodic problem for a Hamiltonian system can be approached through a variational method. A significant challenge in this approach is that the associated functional is highly indefinite. On the other hand, Conley and Zehnder~\cite[Theorem 3]{ConZeh1983IM} proved another version of the Poincar\'e--Birkhoff Theorem in higher dimensions by using variational methods and the Conley index theory for flows. Their result concerns the multiplicity of periodic solutions for time-dependent Hamiltonian vector fields provided that the Hamiltonian function $H=H(t,x,y)$ with $x=(x_1, \ldots, x_N)$, $y=(y_1, \ldots, y_N)$ is twice continuously differentiable, periodic in $t$ and in the variables $x_i$, and quadratic in $y$ on a neighborhood of infinity. Then, they obtained the existence of at least $N+1$ periodic solutions. Remarkably, their result does not rely on the Poincar\'e time-map being close to the identity or having a monotone twist.

\medbreak

Inspired by the idea of Conley and Zehnder, Chang~\cite{Cha1989NA} provided an alternative proof of their result, highlighting that the periodicity of the Hamiltonian function $H$ is fundamental to the existence of multiple periodic solutions. The advancement of infinite-dimensional Lusternik--Schnirelmann methods enabled Szulkin~\cite{Szu1990NA} to extend the Conley--Zehnder theorem, making it applicable to a broader range of Hamiltonian systems. Some other results concerning periodic solutions of Hamiltonian systems can be found in~\cite{Fel1992JDE, Jos1994PLMS}.

\medbreak

\medbreak

Taking a further step in this direction Fonda and Ure\~na~\cite{FonUre2016CRASP, FonUre2017APAN} proposed a new version of higher-dimensional Poincar\'e--Birkhoff Theorem which apply to Poincar\'e time-maps of Hamiltonian systems. Later on, Fonda and Gidoni~\cite{FonGid2020NODEA} extended this result for Hamiltonian systems coupling twisting components with nonresonant linear components. The coupling of twisting components with {\em resonant }linear components was studied by Chen and Qian~\cite{CheQia2022JDE}, where they used an Ahmad-Lazer-Paul type condition to provide a multiplicity result. Some further extensions of higher-dimensional Poincar\'e--Birkhoff Theorem to different coupled Hamiltonian systems can be found in~\cite{FonGarSfe2023JMAA, FonMamSfe2024RM, FonUll2024JDE, FonUll2024DIE, MamUll2024Pre}.

\medbreak

On the other hand, recently, Fonda and Ortega~\cite{FonOrt2023RCMPS} proved a multiplicity result for a two-point boundary value problem associated with Hamiltonian systems on a cylinder. Unlike the periodic problem, where the Poincar\'e--Birkhoff Theorem plays a central role, they did not use any twist condition. Similar to the periodic case, there are generalizations of this result for coupled Hamiltonian systems, see for example~\cite{FonMamObeSfe2024NODEA, Mam2024TMNA, FonUll2024NODEA}. While the Neumann problem for scalar second order equations has been widely studied, few papers in the literature provide multiplicity results for Neumann-type problems associated with systems of ordinary differential equations. To the best of our knowledge, multiplicity results for mixed periodic and Neumann-type problems associated with Hamiltonian systems have never been studied before.

\medbreak
It is the aim of this paper to study Hamiltonian systems by coupling two systems in which one is associated with periodic and the other with Neumann-type boundary conditions. To be more precise, we consider the Hamiltonian system
\begin{equation}\label{sys:main}
\begin{cases}
\dot{q}= \nabla_{p} H(t,q,p,u,v)\, ,\\
\dot{p}= -\nabla_{q} H(t,q,p,u,v)\,, \vspace{1mm}\\
\dot{u}= \nabla_v H(t,q,p,u,v) \,,\\
\dot{v} = - \nabla_u H(t,q,p,u,v) \,,
\end{cases}
\end{equation}
with the boundary conditions
\begin{equation}\label{eq:of:npbc}
    \begin{cases}
            q(a)=q(b)\,, \quad p(a) = p(b)\,,\\
            v(a)=0=v(b)\,.
    \end{cases}
\end{equation}
We write 
\begin{align*}
&q=(q_1 , \dots , q_M) \in\R^{M}, \quad p=(p_1 , \dots , p_M)\in\R^{M},\\
&u=(u_1 , \dots , u_{L}) \in\R^{L}, \quad v=(v_1 , \dots , v_{L}) \in\R^{L}\, .
\end{align*}
The function $H:[a,b] \times \R^{2M} \times\R^{2L} \to \R$, is continuous, and continuously differentiable with respect to $q,p,u$ and $v$.

\medbreak

Here are our hypotheses.

\medbreak

\ni $A1$. There exist $\kappa_i>0$ for $i=1, \dots, M$ and $\tau_j>0$ for $j=1, \dots, L$ such that the function $H(t,q,p,u,v)$ is $\kappa_i$-periodic in the variable~$q_i$ and is $\tau_j$-periodic in the variable~$u_j$.



\medbreak

We now introduce the {\em twist condition} by first considering the case when $\mathcal{D}$ is a rectangle in $\mathbb{R}^M$, i.e.,
$$
\mathcal{D}= [c_1, d_1] \times \cdots \times [c_M,d_M]\,,
$$
and we denote by $\mathring{\mathcal D}$ its interior.

\medbreak

\noindent $A2$. There exists an $M$-tuple $\sigma=(\sigma_1 , \dots , \sigma_M) \in \{-1,1\}^{M}$ such that for every $C^1$-function $\mathcal{W}:[a,b] \to \R^{2L}$, all the solutions $(q,p)$ of system
\begin{equation}\label{eq:of:hamiltonian:unperterbed:for:higher}
\begin{cases}
    \dot{q}= \nabla_{p} H(t,q,p,\mathcal{W}(t))\\
    \dot{p}= -\nabla_{q} H(t,q,p,\mathcal{W}(t))    
\end{cases} 
\end{equation}
starting with $p(a)\in{\cal D}$ are defined on $[a,b]$, and for every $i=1, \dots, M$ we have
$$
\begin{cases}
p_{i}(a)=c_i \quad\Rightarrow\quad \sigma_{i}(q_{i}(b)-q_{i}(a)) < 0\,,\\
p_{i}(a)=d_i\quad\Rightarrow\quad \sigma_{i}(q_{i}(b)-q_{i}(a)) > 0\,.
\end{cases}
$$

\medbreak

\ni $A3$. All the solutions of system~\eqref{sys:main} satisfying $v(a)=0$ are defined on $[a,b]$. 

\medbreak


\medbreak

Here is our main result.

\begin{Thm}\label{thm:main}
Assume that $A1$--$A3$ hold true.
Then there are at least $M+L+1$ geometrically distinct solutions of the boundary value problem~\eqref{sys:main}--\eqref{eq:of:npbc} such that $p(a) \in \mathring{\mathcal D}$.
\end{Thm}

 Notice that, when a solution has been found, infinitely many others appear by just adding an integer multiple of $\kappa_i$ to the $q_i$-th component or adding an integer multiple of $\tau_j$ to the $u_j$-th component. We say that two solutions are {\em geometrically distinct} if they cannot be obtained from each other in this way.

The rest of the paper is organized as follows.

\medbreak

In Section~\ref{sec:2} we provide a proof of our result for the low dimensional case by taking $L=M=1$. The proof in higher dimension is given in Section~\ref{sec:3}.  Finally, in Section~\ref{sec:4} we provide a possible example of application of the main result, we add some other possible variants of {\em twist assumptions}, and conclude with a discussion of some open problems in this direction.

\section{Proof of Theorem~\ref{thm:main}\label{sec:2} in low dimension}
We discuss the case when $M=L=1$, and use the notations $\kappa=\kappa_1$ and $\tau=\tau_1$ for simplicity. The proof follows a variational approach and is based on a theorem by Szulkin which will be recalled later. To simplify the proof, we assume that the Hamiltonian function $H(t,q,p,u,v)$ is $C^1$ in all variables just to follow the idea in~\cite{FonUre2016CRASP}. One can prove the same result by assuming $H(t,q,p,u,v)$ to be continuous and $C^1$ in $(q,p,u,v)$, but then should follow the proof in~\cite{FonUre2017APAN}, which is more technical.

\medbreak

Without loss of generality, we may assume that $[a,b]=[0, \pi]$. By $A2$, $A3$ and a standard compactness argument \cite[Appendix]{FonSfeToa2024Pre}, there exists a constant $C>1$ such that, for any solution $(q,p,u,v)$ of~\eqref{sys:main} satisfying $p(0) \in [c, d]$ and $v(0)=0$, one has that
\begin{equation*}
|p(t)| \le C\,, \quad |v(t)| \le C\,,\quad\hbox{ for every } t \in [0,\pi]\,.
\end{equation*} 
Let $\sigma:\R \to\R$ be a $C^{\infty}$-function such that
\begin{equation}\label{eq:of:sigma:function} 
\sigma(s)=
\begin{cases} 
1\,, &\text{if } \; |s| \le C \,, \\
0\,, & \text{if } \; |s| \ge C  +1\,,
\end{cases}
\end{equation}
and set 
\begin{equation*}
\widehat{H}(t,q,p,u,v)= \sigma(|p|)\sigma(|v|) H(t,q,p,u,v)\,.
\end{equation*}
Now consider the modified system
\begin{equation}\label{sys:main:modified:1}
\begin{cases}
\dot{q}= \partial_{p} \widehat{H}(t,q,p,u,v)\, ,\\
\dot{p}= -\partial_{q} \widehat{H}(t,q,p,u,v)\,, \vspace{1mm}\\
\dot{u}= \partial_v \widehat{H}(t,q,p,u,v) \,,\\
\dot{v} = - \partial_u \widehat{H}(t,q,p,u,v) \,,
\end{cases}
\end{equation}
so that the partial derivatives of $\widehat{H}(t,q,p,u,v)$ with respect to $q,p,u$ and $v$ are bounded. In particular, there exists $\bar c>0$ such that 
\begin{equation} \label{eq:of:bddness:of:hatH}
\left| \frac{\partial \widehat{H}}{\partial p} \right| < \bar c\,,    
\end{equation}

\medbreak
Without loss of generality, we assume that $[c,d]=[-1,1]$, and for $\zeta=(\xi, \eta, \mu, \nu) \in \R^{4}$, we denote by $\mathcal{Z}(t, \zeta)$ the value of the solution $z=(q,p,u,v)$ of system~\eqref{sys:main:modified:1} with $z(0)= \zeta$. The function $\mathcal{Z}: \R \times \R^2 \times \R^2 \to \R^2 \times \R^2 $ is globally defined. For any $t,$ we consider the flow $ \mathcal{Z}_t := \mathcal{Z}(t; \cdot): \R^2  \times \R^2 \to \R^2 \times \R^2$, and denote the corresponding coordinates by $ \mathcal{Q}_{t}, \mathcal{P}_{t}, \mathcal{U}_{t}, \mathcal{V}_{t} : \R^4 \to \R $, i.e., $\mathcal{Z}_{t}= ( \mathcal{Q}_{t}, \mathcal{P}_{t}, \mathcal{U}_{t}, \mathcal{V}_{t})$. Based on our assumptions, the following assertions hold true:

\medbreak
 {\ni (i)} The map $\mathcal{Z}_{0}$ is identity on $\R^4$.
 \medbreak

 {\ni (ii)} $\mathcal{Z}_t((\xi, \eta, \mu, \nu) + h) = \mathcal{Z}_t(\xi, \eta, \mu, \nu)+h $ for every $h \in \kappa \Z \times \{0\} \times \tau \Z \times \{0\}$.
 \medbreak
{\ni (iii)} Each $\mathcal{Z}_t$ is a canonical transformation of $\R^4$ on itself. 
\medbreak

{\ni (iv)} $\mathcal{Z}(t;\xi, \eta, \mu, \nu)= (\xi, \eta, \mu, \nu)$ if $ |\eta| \geq C+1$ or $|\nu| \geq C+1$.

\medbreak

Using~\eqref{eq:of:bddness:of:hatH} and the Mean Value Theorem, we have
\begin{equation}\label{eq:of:appl:of:mean:val}
|\mathcal{Q}_{\pi}( \xi, \eta, \mu, \nu)-  \xi| < \bar c \pi\,.    
\end{equation}
We take $\sigma_1=1$ in assumption $A2$ (the case $\sigma_1=-1$ can be use alternatively), and
thus by Ascoli--Arezl\`a Theorem there exists $\epsilon \in\, ]0,1[$ such that
\begin{equation}\label{eq:after:Asc:arz}
    \begin{cases}
 \mathcal{Q}_{\pi}( \xi, \eta, \mu, 0)-   \xi < 0 \quad \hbox{ for every } \xi,  \mu \in \R, \; \eta \in [-1-\epsilon, -1]\,,\\
\mathcal{Q}_{\pi}( \xi, \eta, \mu, 0) -  \xi > 0 \quad \hbox{ for every } \xi, \mu \in \R, \; \eta \in [1, 1+ \epsilon] \,.
    \end{cases}
\end{equation}
We choose a $C^{\infty}$-function $\gamma:[0, \infty[ \to \R$ with
\begin{equation}\label{eq:of:gamma}
    \begin{cases}
        \gamma(s)= 0 \hbox{ for } s \in [0,1]\,, \quad \gamma'(s) \geq 0 \hbox{ for } s \in\, ]1,1+ \epsilon[\,,\\
        \gamma'(s) \geq 1 \hbox{ for } s \in [1+ \epsilon, 2]\,, \quad \gamma(s)= s^2 \hbox{ for } s \in [ 2,+ \infty[\,,
    \end{cases}
\end{equation}
and let $\lambda >1$ be a parameter, to be fixed later. Define a function $ \mathscr{R}_\lambda: \R^{4} \to \R$ by 
$$
\mathscr{R}_\lambda(\xi, \eta, \mu, \nu):=  \lambda \gamma(|\eta|)\,,
$$
and the function $ R_\lambda: \R \times \R^4 \to \R$ by
$$
R_{\lambda}(t, \cdot)= \mathscr{R}_\lambda \circ \mathcal{Z}_t^{-1}, \quad \hbox{ if } 0 \leq t \leq \pi\,.
$$
Now set 
$$
\widetilde{H}(t, z)= \widehat{H}(t,z) + R_{\lambda}(t,z)\,,
$$
and consider the modified system
\begin{equation}\label{sys:main:modified:2}
\begin{cases}
\dot{q}= \partial_{p} \widetilde{H}(t,q,p,u,v)\, ,\\
\dot{p}= -\partial_{q} \widetilde{H}(t,q,p,u,v)\,, \vspace{1mm}\\
\dot{u}= \partial_v \widetilde{H}(t,q,p,u,v) \,,\\
\dot{v} = - \partial_u \widetilde{H}(t,q,p,u,v) \,.
\end{cases}
\end{equation}
The function $\widetilde{H} : [0, \pi] \times \R^4 \to \R$ satisfies the following properties.

\medbreak

\ni $(a)$ $\widetilde{H}(t, z)  = \widetilde{H}(t, z+h)$ for every $h \in \kappa \Z \times \{0\} \times \tau \Z \times \{0\}$.
\medbreak
\ni $(b)$ $\widetilde{H}(t, z) = \lambda |p|^2$ if $ |p| \geq C+1$.

\medbreak
\ni $(c)$ $\widetilde{H}$ and $\widehat{H}$ coincide on the set 
$$B=\{ (t, \mathcal{Z}(t, \xi, \eta, \mu, \nu))  :  t \in [ 0, \pi], \; \eta \in [-1,1]\}\,.$$

\medbreak

We will now introduce the function spaces and the needed functional.

\subsection{The function spaces}
For the variables $q$ and $p$, we consider the space $H_{\pi}^\frac{1}{2}(0, \pi)$, whose elements are those real valued functions $f$ in $L^{2}(0, \pi)$ with the property that, writing the associated Fourier series
$$
f(t) \sim \sum_{ k \in \Z} f_k e^{2k i t}\,,
$$
one has that
$$
 \sum_{k \in \Z} (1+ |k|) |f_k|^2 < +\infty\,.
$$
This space $H_{\pi}^\frac{1}{2}$ is a Hilbert space with the scalar product given by
$$
\left\langle \sum_{ k \in \Z} f_k e^{2k i t}, \sum_{ k \in \Z} g_k e^{2k i t} \right\rangle_\frac{1}{2}= f_0 g_0^{*} + \sum_{ 0 \neq k \in \Z} |k| f_k g_{k}^{*} \,,
$$
where $g_{k}=g_{-k}^{*} \in \mathbb{C}$ and $f_{k}=f_{-k}^{*} \in \mathbb{C}$.

\medbreak
On the other hand, for the variables $u$ and $v$, we use the spaces constructed in~\cite{FonOrt2023RCMPS} as follow.

\medbreak

For any $\alpha \in\,]0,1[$\,, we define $X_{\alpha}$ as the set of those real valued functions $\tilde{u}\in L^{2}(0,\pi)$ such that 
\begin{equation*}
	\tilde{u}(t)\sim \sum\limits_{m=1}^{\infty}\tilde{u}_{m}\cos(mt)\,,
\end{equation*} 
where $(\tilde{u}_{m})_{m\ge 1}$ is a sequence in $\R$  satisfying
\begin{equation*}
	\sum\limits_{m=1}^{\infty}m^{2\alpha}\tilde{u}_{m}^{2}<\infty\,.
\end{equation*}
The space $X_{\alpha}$ is endowed with the inner product and the norm
\begin{equation*}
	\langle \tilde{u},\, \tilde{\phi}\rangle_{X_{\alpha}}=\sum\limits_{m=1}^{\infty}m^{2\alpha}\tilde{u}_{m}\tilde{\phi}_{m}\,,
	\qquad
	\|\tilde{u}\|_{X_{\alpha}}=\sqrt{\sum\limits_{m=1}^{\infty}m^{2\alpha}\tilde{u}_{m}^{2}}\,.
\end{equation*}
For any $\beta \in\,]0,1[$\,, we define $Y_{\beta}$ as the set of those real valued functions $v\in L^{2}(0,\pi)$ such that 
\begin{equation*}
	v(t)\sim \sum\limits_{m=1}^{\infty}v_{m}\sin(mt)\,,
\end{equation*} 
where $(v_{m})_{m\ge 1}$ is a sequence in $\R$  satisfying
\begin{equation*}
	\sum\limits_{m=1}^{\infty}m^{2\beta}v_{m}^{2}<\infty\,.
\end{equation*}
The space $Y_{\beta}$ is endowed with the inner product and the norm
\begin{equation*}
	\langle v,\, \rho\rangle_{Y_{\beta}}=\sum\limits_{m=1}^{\infty}m^{2\beta}v_{m}\rho_{m}\,,
	\qquad	\|p\|_{Y_{\beta}}=\sqrt{\sum\limits_{m=1}^{\infty}m^{2\beta}v_{m}^{2}}\,.
\end{equation*}
From now on, we will consider functions $q,p,u,v$ which can be written as
\begin{equation}\label{eq:of:notations}
    \begin{cases}
	q(t)=\bar q+ \tilde  q(t)\,, \quad \bar q=\frac{1}{\pi}\int\limits_{0}^{\pi}q(t)\,dt\,,\\
	u(t)=\bar u+\tilde u(t)\,,\quad \bar u=\frac{1}{\pi}\int\limits_{0}^{\pi}u(t)\,dt\,,
 \end{cases}
\end{equation}
where $\tilde q$ and $p$ belong to the spaces $\widetilde{H}_{\pi}^{\frac{1}{2}}$ and $H_{\pi}^{\frac{1}{2}}$ respectively, while functions $\tilde u$ and $v$ belong to the spaces $X_{\alpha}$ and $Y_{\beta}$ respectively. Here, we use the notation $\widetilde{H}_{\pi}^{\frac{1}{2}}$ for the Hilbert space of those $\tilde{u}$ in $H_{\pi}^{\frac{1}{2}}$ whose mean is zero.

\medbreak

Choose two positive numbers $\alpha,\beta$ such that 
$$ 
\alpha <\half<\beta\quad\hbox{ and }\quad\alpha+\beta=1\,.
$$
Consider the torus $\mathbb{T}^{2}=(\R/\kappa\mathbb{Z}) \times (\R/\tau\mathbb{Z})$, and the space $E= E_1 \times E_2$ where $E_1=\widetilde{H}_{\pi}^{\frac{1}{2}} \times H_{\pi}^{\frac{1}{2}}$ and $E_2=  X_{\alpha} \times Y_{\beta} $. The space $E$ is endowed with the scalar product
\begin{align*}	\langle(\tilde{q}\,,p\,,\tilde{u}\,,v),\,(\widetilde{Q}\,,P\,,\widetilde{U}\,,V)\rangle_{E}
=&\langle\tilde{q},\,\widetilde{Q}\rangle_\frac{1}{2}+\langle p,\,P\rangle_\frac{1}{2}+\\
&+  \langle \tilde{u},\,\widetilde{U}\rangle_{X_{\alpha}}+\langle v,\,V\rangle_{Y_{\beta}} \,,
\end{align*} 
and the corresponding norm
\begin{equation*}
\|(\tilde{u}\,,p\,,\tilde{u}\,,v)\|_{E}=\sqrt{\|\tilde{u}\|^{2}_\frac{1}{2}+\|p\|^{2}_\frac{1}{2}+ \|\tilde{u}\|^{2}_{X_{\alpha}}+\|v\|^{2}_{Y_{\beta}}}\,.
\end{equation*}
Since $H_{\pi}^\frac{1}{2}(0, \pi)$, $ X_\alpha$ and $Y_\beta$ are separable Hilbert spaces~\cite[Proposition 2.3 and 2.6]{FonOrt2023RCMPS}, the same is true for $E$. 

\medbreak

By $A1$, the Hamiltonian function $ \widetilde{H}$ is $\kappa$-periodic in $q$ and $\tau$-periodic in $u$, hence using the notations in~\eqref{eq:of:notations}, we can assume that $(\bar q, \bar u)\in \mathbb{T}^{2}$ and thus look for solutions $( w, \bar q, \bar u ) \in E \times \mathbb{T}^{2}$, where
$$
w =(\tilde q,p, \tilde u, v)\,.
$$
These solutions will be found as critical points of a suitable functional, by applying the following theorem of Szulkin~\cite{Szu1990NA} (see also~\cite{FouLupRamWil, Liu1989JDE}).  

\begin{Thm}\label{th:Szulkin}
If $\varphi: E\times \mathbb{T}^{N}\to \R$ is a continuously differentiable functional of the type
\begin{equation*}
\varphi(w ,\bar s)=\frac{1}{2}\langle \mathscr{L}w ,w \rangle_{E}+\psi(w ,\bar s)\,,
\end{equation*}
where $\mathscr{L}:E\to E$ is a bounded selfadjoint invertible operator and $d\psi(E\times \mathbb{T}^{N})$ is relatively compact, then $\varphi$ has at least $N+1$ critical points.
\end{Thm}

\subsection{The functional and the bilinear form}

We define a functional $\psi: E \times \mathbb{T}^{2}\to \R$ as 
\begin{align*}
	\psi(w ,\bar q, \bar u) &= \psi \big(  (\tilde q\,,p \,, \tilde u, v\,, \bar q,  \bar u\big)\\
	&= \int_0^\pi 
	\big[\widetilde{ H} \big(t\,,\bar q+\tilde q(t)\,,p(t)\,,\bar u+\tilde u(t)\,,v(t)\big) - \lambda |p(t)|^2\big]
	\,dt\,.
\end{align*}
In the following, we will treat $\mathbb{T}^{2}$ as being lifted to $\R^{2}$, so $E\times \mathbb{T}^{2}$ will often be identified with $E\times\R^{2}$\,. 
Similar to the proof provided in~\cite[Proposition 2.10]{FonOrt2023RCMPS} and~\cite[Proposition 19, Proposition 22]{FonMamObeSfe2024NODEA}, we can show that $\psi$ is continuously differentiable, and the gradient function $\partial \psi $ has a relatively compact image.
In what follows we introduce the operator $\mathscr{L}$.

\medbreak

Denote by $\widetilde C^{1}([0,\pi])$, the space of those functions in $ C^{1}([0,\pi])$ which have mean zero, consider the space
$$
D= \widetilde C^{1}([0,\pi])\times C^{1}([0,\pi]) \times \widetilde C^1([0,\pi])\times C^1_0([0,\pi]) \,,
$$
and define a bilinear form $\BB : D \times D \to \R$ as follows.
For every 
$w =(\tilde q, p  , \tilde u, v)$ and $W =(\widetilde{Q}, P , \widetilde U, V)$ in $D$, 
\begin{multline}\label{eq:of:bilinear}
\BB(w ,W) =\int_0^\pi \Big[ 
	\langle p'(t),\widetilde{Q}(t) \rangle  -\langle \tilde q'(t) ,P(t)\rangle + 2 \lambda \langle p(t), P(t) \rangle\\
 +\langle v'(t),\widetilde{U}(t) \rangle  -\langle \tilde{u}'(t) ,V(t)\rangle
	\Big]\,dt\,.    
\end{multline}
It is immediate to see that the bilinear form $\BB$ is symmetric.

\begin{pro}\label{pro:of:biliear:form}
The set $D$ is a dense in $E$, and the bilinear form $\BB : D \times D \to \R$ is continuous with respect to the topology of $E\times E\,$.  
\end{pro}

\begin{proof}
We know by~\cite[Proposition 2.5 and 2.8]{FonOrt2023RCMPS} that $ \widetilde C^1([0,\pi])\times C^1_0([0,\pi])$ is a dense subspace of $E_2$, and being $\widetilde C^{1}([0,\pi])\times C^{1}([0,\pi])$ a dense subspace of $E_1 $, we can say that $D$ is a dense subspace of $E$. In order to prove the second part of the statement, let us write
$$
\BB(w ,W)= \BB_1 \big((\tilde{q},p), (\widetilde{Q},P)\big) + \BB_2\big((\tilde u, v), (\widetilde U, V)\big)\,,
$$
where 
\begin{equation}\label{eq:of:b1}
 \BB_1 \big((\tilde{q},p), (\widetilde{Q},P) \big)= \int_0^\pi \Big(
	\langle p'(t),\widetilde{Q}(t) \rangle - \langle \tilde q'(t),P(t)\rangle+ 2 \lambda \langle p(t), P(t) \rangle  \Big)dt \,,    
\end{equation}
and 
\begin{equation}\label{eq:of:b2}
    \BB_2\big((\tilde u, v), (\widetilde U, V)\big)= \int_0^\pi \Big(\langle v'(t),\widetilde U(t) \rangle  -\langle \tilde u'(t) ,V(t)\rangle \Big)dt \,.
\end{equation}
It has been proved in~\cite{HofZeh1994Bir} and~\cite[Section 3.4]{FonMamObeSfe2024NODEA} that $\BB_1$ and $\BB_2$ are continuous respectively with respect to the topology of $E_1$ and $E_2$. This completes the proof.
\end{proof}

\medbreak

The bilinear form $\BB : D \times D \to \R$ can thus be extended in a unique way to a continuous symmetric bilinear form $\BB : E \times E \to \R$, for which we maintain the same notation. A bounded selfadjoint operator $\mathscr{L}:E \to E$ can thus be defined by 
\begin{equation*}
  \langle \mathscr{L} w , W\rangle_{E}=\BB(w ,W)\,, 
\end{equation*}
for $w, W$ in $E$. Referring to~\eqref{eq:of:b1} and~\eqref{eq:of:b2}, we can write
 $$
 \mathscr{L}(\tilde q, p , \tilde u, v)=(\mathscr{L}_{1}(\tilde q, p ), \mathscr{L}_{2}(\tilde u, v))\,,
 $$
 where 
$$
\langle \mathscr{L}_{1}(\tilde q, p ), (\widetilde{Q},P)  \rangle_{E_1}=\BB_1 \big((\tilde{q},p), (\widetilde{Q},P) \big)\,, 
$$
and 
$$
\langle  \mathscr{L}_{2}(\tilde u, v), (\widetilde{U}, V)  \rangle_{E_2}=\BB_2((\tilde u, v), (\widetilde{U}, V)) \,,
$$
for every $w =(\tilde q, p  , \tilde u, v)$ and $W =(\widetilde{Q}, P ,  \widetilde U, V)$ in $E$.
It has been proved in~\cite[Proposition 2.14]{FonOrt2023RCMPS} that
\begin{equation}\label{eq:of:norm:of:l2}
    \|\mathscr{L}_{2}(\tilde{u}, v)\|_{E_2}= \frac{\pi}{2} \|(\tilde{u}, v)\|_{E_2}\,.
\end{equation}

\medbreak

We now need the following result.

\begin{lem}\label{lem:of:L1:greater:than}
For every $(\tilde{q}, p) \in E_1,$ we have
\begin{equation*}\label{ineq:of:L1:greater:than}
 \|\mathscr{L}_1(\tilde{q}, p)\|_{E_1} \geq 2 \pi \|(\tilde{q}, p)\|_{E_1}\,.   
\end{equation*}
\end{lem}

\begin{proof}
Let $\mathscr{L}_1(\tilde{q}, p) = \delta $ for $\delta = (\tilde{x}, y)$. Then we have 
$$
 \langle \delta, (\widetilde{Q}, P) \rangle_{E_1}=B_1((\tilde{q}, p), (\widetilde{Q}, P)),
$$
for every $(\widetilde{Q}, P) \in E_1$. If $(\widetilde{Q}, P) \in \widetilde C^{1}([0,\pi])\times C^{1}([0,\pi])$, then we have
\begin{multline}\label{eq:of:comparison}
\langle \tilde{x}, \widetilde{Q} \rangle_{\frac12} + \langle y, P \rangle_{\frac12}=
\int_0^\pi \langle p'(t), \widetilde{Q}(t) \rangle dt - \int_0^\pi \langle \tilde{q}'(t), P(t) \rangle dt\\ + 2 \lambda  \int_0^\pi \langle p(t),  P(t) \rangle dt.
\end{multline}
Using the notations
$$
 \tilde{q}(t) = \sum_{0 \ne k \in \Z } \tilde{q}_k e^{i 2k t}, \quad p(t) = \sum_{k \in \mathbb{Z}} p_k e^{i 2k t},
$$
$$
 \tilde{x}(t) = \sum_{0 \ne k \in \mathbb{Z}} \tilde{x}_k e^{i 2k t}, \quad y(t) = \sum_{ k \in \mathbb{Z}} y_k e^{i 2k t},
$$
$$
 \widetilde{Q}(t) = \sum_{0 \ne k \in \Z} \widetilde{Q}_k e^{i 2k t}, \quad  P(t) = \sum_{ k \in \mathbb{Z}} P_k e^{i 2k t},
$$
and putting $P = 0$ in \eqref{eq:of:comparison}, we have
$$
 \sum_{0 \neq k \in \Z}  |k| \langle \tilde{x}_k,  \widetilde{Q}_k \rangle=  2\pi \sum_{0 \neq k \in \Z}  ik \langle p_k, \widetilde{Q}_k \rangle\,,
$$
which implies that
\begin{equation}\label{eq:of:tilte:xk}
\tilde{x}_k =  \frac{2\pi ki}{|k|} p_k\,, \quad \hbox{ for } k\ne 0\,.    
\end{equation}
On the other hand, taking $\widetilde{Q} = 0$ in~\eqref{eq:of:comparison}, we obtain that
$$
\langle y_0, P_0 \rangle + \sum_{0 \neq k \in \Z}  |k| \langle y_k,  P_k \rangle=  2\pi \sum_{0 \neq k \in \Z}   \langle -ki\tilde{q}_k +  \lambda p_k, P_k \rangle+ 2 \lambda \pi \langle p_0, P_0 \rangle\,,
$$
which implies that
\begin{equation}\label{eq:of:yk}
\begin{cases}
    y_0 = 2 \lambda \pi p_0\\
    y_k =  2\pi\frac{- ki\tilde{q}_k +  \lambda p_k }{|k|}\,, \quad \hbox{ for } k\ne 0\,.
\end{cases}    
\end{equation}
Combining~\eqref{eq:of:tilte:xk} and~\eqref{eq:of:yk} together with the definition of norm on $E_1$, we have
\begin{align*}
\|\tilde{x}\|^2_{\frac12} +\|y\|^{2}_{\frac12}  &= 4 \pi^2 \lambda^2 |p_0|^2 +4 \pi^2 \sum_{0 \neq k \in \Z } \frac{(|k|^2+ \lambda^2)}{|k|} |p_k|^2 + 4 \pi^2\sum_{0 \neq k \in \Z } |k| |\tilde{q}_k|^2  \\ 
&\geq 4 \pi^2 \Big( |p_0|^2+ \sum_{0 \neq k \in \Z } |k| |p_k|^2 + \sum_{0 \neq k \in \Z } |k| |\tilde{q}_k|^2 \Big)\\
&= 4 \pi^2 \big(   \|\tilde{q}\|^2_{\frac12}+ \|p\|^2_{\frac12}\big) \,,
\end{align*}
which completes the proof.
 \end{proof}
As a consequence of~\eqref{eq:of:norm:of:l2} and Lemma~\ref{lem:of:L1:greater:than}, $\hbox{ker } \mathscr{L}=\{0\}$. Moreover, by a classical reasoning~\cite[Proposition 2.14]{FonOrt2023RCMPS}, we conclude that the operator $\mathscr{L}$ is bijective and its inverse $\mathscr{L}^{-1}:E \to E$ is continuous. 

\medbreak
By Theorem~\ref{th:Szulkin}, we conclude that the functional $\varphi$ has at least three critical points. The next result shows that these critical points correspond to solutions of the problem~\eqref{sys:main:modified:2}--\eqref{eq:of:npbc}.

\begin{pro}\label{pro:of:critical:points}
If $(w_0, \bar q_{0}, \bar u_{0})=(\tilde{q}_{0}, p_{0},  \tilde{u}_{0}, v_{0},\bar q_{0}, \bar u_{0} )$ is a critical point of the functional $\varphi$, then $z_0(t)=(\bar q_{0}+\tilde{q}_{0}(t), p_{0}(t),  \bar u_{0}+\tilde{u}_{0}(t), v_{0}(t) )$ is a solution of the problem~\eqref{sys:main:modified:2}--\eqref{eq:of:npbc}.  
\end{pro}

\begin{proof}
  Since $(w_0, \bar q_{0}, \bar u_{0}) \in E \times \R^2$ is a critical point of the functional $\varphi$, for any $(w, \bar q, \bar u) = (\tilde{q}, p,  \tilde{u}, v,\bar q, \bar u )\in E \times \R^2$, we have
\begin{equation}\label{eq:of:critical:point:of:phi}
    0= d\varphi(w_0, \bar q_{0}, \bar u_{0})(w, \bar q, \bar u) = \BB(w_0, w)+ d \psi (w_0, \bar q_{0}, \bar u_{0})(w, \bar q, \bar u)\,.
\end{equation}
Moreover, we have
\begin{align*}
    d & \psi (w_0, \bar q_{0}, \bar u_{0})(w, \bar q, \bar u) = \lim_{s \to 0} \frac{1}{s} \big( \psi(w_0 + sw, \bar q_{0}+ s\bar q, \bar u_{0} +s \bar u )- \psi( w_0, \bar q_{0}, \bar u_{0}) \big) \\
    = & \int_{0}^{\pi} \partial_{q} \widetilde{H}\big(t, z_0(t) \big) (\bar q + \tilde{q}(t))dt  + \int_{0}^{\pi} \big[ \partial_{p}\widetilde{H}\big(t, z_0(t) \big) - 2 \lambda p_0(t) \big] p(t)  dt\\
    &+ \int_{0}^{\pi} \partial_{u} \widetilde{H}\big(t, z_0(t) \big) (\bar u + \tilde{u}(t))dt+ \int_{0}^{\pi} \partial_{v} \widetilde{H}\big(t, z_0(t) \big) v(t)dt\,.
\end{align*}
If we consider $ \tilde{q} \in C^1([0, \pi])$ and choose $(w, \bar q, \bar u)=\big((\tilde{q},0, 0,0),0, 0 \big)$ in~\eqref{eq:of:critical:point:of:phi}, we obtain
$$
0= \int_{0}^{\pi}\big[ p_{0}'(t) \tilde{q}(t) +  \partial_{q} \widetilde{H}\big(t, z_0(t) \big) \tilde{q}(t) \big] dt\,,
$$
which implies that
$$
\int_{0}^{\pi} p_{0}'(t) \tilde{q}(t) dt =- \int_{0}^{\pi}  \partial_{q} \widetilde{H}\big(t, z_0(t) \big)\tilde{q}(t) dt\,.
$$
Therefore, in the sense of distributions, we have
$$
p_{0}'(t) = - \partial_{q} \widetilde{H}\big(t, z_0(t) \big)\,,
$$
which is exactly the second equation in~\eqref{sys:main:modified:2}. In particular, $p_0 \in W^{1,2}(0, \pi)$ and therefore it is continuous.
\medbreak

Now if choose $(w, \bar q, \bar u)=\big((0,p, 0,0),0, 0 \big)$ in~\eqref{eq:of:critical:point:of:phi}, we have
$$
0= \int_{0}^{\pi}\big[ - \tilde{q}_{0}'(t)p(t) + 2 \lambda p_0(t) p(t) + \big[ \partial_{p} \widetilde{H}\big(t, z_0(t) \big) -2 \lambda p_0(t) \big] p(t) \big] dt\,,
$$
which implies that
$$
\int_{0}^{\pi} \tilde{q}_{0}'(t) p(t) dt =\int_{0}^{\pi}  \partial_{p} \widetilde{H}\big(t, z_0(t) \big) p(t) dt\,.
$$
Therefore, in the sense of distributions, we have
$$
q_{0}'(t)=\tilde{q}_{0}'(t) = - \partial_{p} \widetilde{H}\big(t, z_0(t) \big)\,,
$$
which is the first equation in~\eqref{sys:main:modified:2}.

\medbreak

Similarly, if we choose $(w, \bar q, \bar u)=\big((0,0,\tilde{u},0),0, 0 \big)$ in~\eqref{eq:of:critical:point:of:phi}, we obtain
$$
0= \int_{0}^{\pi}\big[ v_{0}'(t) \tilde{u}(t) + \partial_{u} \widetilde{H}\big(t, z_0(t) \big) \tilde{u}(t) \big] dt\,,
$$
which implies that
$$
\int_{0}^{\pi} v_{0}'(t) \tilde{u}(t) dt =- \int_{0}^{\pi}  \partial_{u} \widetilde{H}\big(t, z_0(t) \big) \tilde{u}(t) dt\,,
$$
so that
$$
v_{0}'(t) =  - \partial_{u} \widetilde{H}\big(t, z_0(t) \big)\,,
$$
which is the fourth equation in system~\eqref{sys:main:modified:2}. By a similar reasoning, choosing $(w, \bar q, \bar u)=\big((0,0,0,v),0, 0 \big)$ we see that the function $u_0$ is continuous, and in the sense of distributions satisfy the third equation in system~\eqref{sys:main:modified:2}. By equations in~\eqref{sys:main:modified:2}, we deduce that $z_0(t)  \in [C^1([0, \pi])]^{4}$, and so the equations are satisfied in the classical sense. Since $ q_{0}, p_0 \in H^{\frac{1}{2}}_{\pi}(0, \pi) $ and $v_0 \in Y_{\beta}$, the boundary conditions~\eqref{eq:of:npbc} are also satisfied, and thus $z_0(t)$ is a solution of the problem~\eqref{sys:main:modified:2}--\eqref{eq:of:npbc}.
\end{proof}

\subsection{Back to the original system}

We complete the proof of Theorem~\ref{thm:main} by proving the following result.

\begin{pro}\label{pro:final:low}
    For $\lambda> \bar c$, where $\bar c$ is as in~\eqref{eq:of:bddness:of:hatH}, system~\eqref{sys:main:modified:2} does not have solutions $z(t)=(q(t),p(t),u(t),v(t))$ departing with $p(0) \notin\, ]-1,1[$ satisfying the boundary conditions~\eqref{eq:of:npbc}.
\end{pro}

\begin{proof}
    We assume by contradiction that for such value of $\lambda$, there exists a solution $z =(q,p,u,v):[0, \pi] \to \R^4$ of system~\eqref{sys:main:modified:2} which starts with $p(0) \notin\, ]-1,1[$  and satisfies the boundary conditions~\eqref{eq:of:npbc}. We consider the function $\zeta:[0, \pi] \to \R^4 $ defined by
    $$
    \zeta(t) := \mathcal{Z}_{t}^{-1}(z(t))\,,
    $$
and claim that 
$$
J \dot{\zeta} = \nabla \mathscr{R}_{\lambda}(\zeta)\,.
$$
\noindent
{\sl Proof of the Claim.} Differentiating $z(t) = \mathcal{Z}_{t}(\zeta(t))$, we obtain
$$
\dot{z}= \frac{\partial \mathcal{Z} }{ \partial t} (t, \zeta) +  \frac{\partial \mathcal{Z} }{ \partial \zeta} (t, \zeta) \dot{\zeta}, 
$$
and so
\begin{equation*}
    \frac{\partial \mathcal{Z} }{ \partial \zeta} (t, \zeta) \dot{\zeta} =  \dot{z}- \frac{\partial \mathcal{Z} }{ \partial t} (t, \zeta)= -J \nabla_{z} \widetilde{H}(t,z) + J \widehat{H}(t,z)= -J \nabla R_{\lambda}(\zeta)\,.
\end{equation*}
By $(iii)$, since $\mathcal{Z}_t$ is canonical, we obtain
$$
\dot{\zeta}=-J \nabla \mathscr{R}_{\lambda}(\zeta)\,,  
$$
for details see \cite[page 9]{FonUre2016CRASP}.

\medbreak
By the claim and the definition of $\mathscr{R}_{\lambda}$, for $\zeta=(\xi, \eta, \mu, \nu)$ we have
$$
\dot{\xi}= \lambda \gamma'(|\eta|) \frac{\eta}{|\eta|}, \quad \dot{\eta}= \dot{\mu}= \dot{\nu}=0\,,
$$
and consequently, by using the fact that $\mathcal{Z}_0$ is identity, we obtain
$$
\eta(t)=\eta(0)=p(0), \quad \nu(t)=\nu(0)=0, \quad \mu(t)= \mu(0)=u(0)\,,
$$
for every $t \in [0, \pi]$, and
$$
\xi(t)= q(0)+ \lambda t \gamma'(|p(0)|) \frac{p(0)}{|p(0)|} \,.
$$
Thus we have
\begin{equation}\label{eq:of:q:of:pi}
q(\pi)= \mathcal{Q}_{\pi}\Big(q(0)+ \lambda \pi\gamma'(|p(0)|) \frac{p(0)}{|p(0)|},\; p(0), u(0), 0 \Big)\,. 
\end{equation}

In order to get a contradiction, we show that $q(\pi) \neq q(0)$, and for this we discuss the following cases.

\medbreak

\ni{\it Case 1.} If $1 \leq |p(0)| < 1+ \epsilon$, then by~\eqref{eq:of:gamma} $\gamma'(|p(0)|) \geq 0$, so by~\eqref{eq:after:Asc:arz} and~\eqref{eq:of:q:of:pi} we have
$$
q(\pi)- q(0)   < - \lambda \pi \gamma'(|p(0)|)   \qquad \hbox{ for } p(0) \in [-1-\epsilon, -1]\,,
$$
and
$$
q(\pi)- q(0)   >  \lambda \pi \gamma'(|p(0)|) \qquad \hbox{ for } p(0) \in [1, 1+\epsilon]\,,
$$
which imply that $q(\pi) \ne q(0)$ in this case.

\medbreak
\ni{\it Case 2.} If $1+ \epsilon \leq |p(0)| < C+1$, then since $\gamma'(|p(0)|) \geq 1$ by~\eqref{eq:of:gamma}, by triangular inequality, we have
$$
|q(\pi)-q(0)| \geq \lambda \pi \gamma'(|p(0)|)- \left|q(\pi)-q(0)+ \lambda \pi \gamma'(|p(0)|) \frac{p(0)}{|p(0)|}\right|,
$$
which by~\eqref{eq:of:appl:of:mean:val} and~\eqref{eq:of:q:of:pi} imply that
$$
|q(\pi)-q(0)|    > \lambda \pi - \bar{c} \pi >0\,,
$$
since $\lambda > \bar{c}$.
\medbreak
\ni{\it Case 3.} If $|p(0)| \geq C+1$, then by~\eqref{eq:of:gamma}, $\gamma'(|p(0)|)= 2 |p(0)|$, and (iv) implies that the map $\mathcal{Q}_{\pi}( \xi, \eta, \mu, \nu)= \xi$. Thus~\eqref{eq:of:q:of:pi} implies that 
$$
q(\pi)= q(0)+2\lambda \pi p(0)\,,
$$
so that $q(\pi) \ne q(0)$ in this case too.
\medbreak
Thus we proved that for $\lambda> \bar c$, system~\eqref{sys:main:modified:2} does not have solutions $z(t)=(q(t),p(t),u(t),v(t))$ departing with $p(0) \notin\, ]-1,1[$ satisfying the boundary conditions~\eqref{eq:of:npbc}.
\end{proof}
\medbreak

By the above property $(c)$, we conclude that problem~\eqref{sys:main}--\eqref{eq:of:npbc} has at least three geometrically distinct solutions, and this completes the proof of the main theorem in the case of low dimension. \qed

\section{Proof of Theorem~\ref{thm:main} in higher dimension}\label{sec:3}
Following the idea of Section~\ref{sec:2}, we first modify the Hamiltonian function by using a cut-off function $\sigma$ defined as~\eqref{eq:of:sigma:function} and set
\begin{equation*}
\widehat{H}(t,q,p,u,v)= \sigma(|p|)\sigma(|v|) H(t,q,p,u,v)\,,
\end{equation*}
so that the gradient of $\widehat{H}(t,q,p,u,v)$ with respect to $q,p,u$  and $v$ are bounded. In particular there exist $\ov{C}>0$ such that
\begin{equation}\label{eq:of:bddness:of:nabla:h}
    \big| \nabla_p \widehat{H} \big| <  \ov{C}\,.
\end{equation}
Without loss of generality, we may take $\mathcal{D}=[-1,1]^M$, and for $\zeta=(\xi, \eta, \mu, \nu) \in \R^{2M+2L}$, we denote $\mathcal{Z}(t, \zeta)$ the value of the solution $z=(q,p,u,v)$ of the modified system with $z(0)= \zeta$. Based on our assumptions, the assertions $(i)$--$(iv)$ hold true in this case too. Proceeding similarly as in Section~\ref{sec:2}, we choose a function like $\gamma$ given in~\eqref{eq:of:gamma}, and define $ \mathscr{R}_\lambda: \R^{2M+2L} \to \R$ by 
$$
\mathscr{R}_\lambda(\xi, \eta, \mu, \nu):= - \lambda \gamma(|\eta|)\,,
$$
and the function $ R_\lambda: \R \times \R^{2M+2L} \to \R$ by
$$
R_{\lambda}(t, \cdot)= \mathscr{R}_\lambda \circ \mathcal{Z}_t^{-1}, \quad \hbox{ if } 0 \leq t \leq \pi\,.
$$
Now set 
$$
\widetilde{H}(t, z)= \widehat{H}(t,z) + R_{\lambda}(t,z)\,,
$$
and consider the modified system
\begin{equation}\label{sys:main:modified:2:higher}
\begin{cases}
\dot{q}= \nabla_{p} \widetilde{H}(t,q,p,u,v)\, ,\\
\dot{p}= -\nabla_{q} \widetilde{H}(t,q,p,u,v)\,, \vspace{1mm}\\
\dot{u}= \nabla_v \widetilde{H}(t,q,p,u,v) \,,\\
\dot{v} = - \nabla_u \widetilde{H}(t,q,p,u,v) \,.
\end{cases}
\end{equation}
The function $\widetilde{H} : [0, \pi] \times \R^{2M+2L} \to \R$ satisfies the properties $(a)-(c)$ with $ \eta \in \mathcal{D}$ in this case.

\medbreak

We consider that functions $\tilde{q}$ and $p$ belong to the spaces $\Big(\widetilde{H}_{\pi}^{\frac12}\Big)^{M}$ and $\Big(H_{\pi}^{\frac12}\Big)^{M}$ respectively, while functions $\tilde u$ and $v$ belong to the spaces $X_{\alpha}^{L}$ and $Y_{\beta}^{L}$ respectively. In addition, we consider the torus 
$$
\mathbb{T}^{M+L}= (\R/\kappa_1\mathbb{Z}) \times \dots \times (\R/\kappa_M\mathbb{Z}) \times (\R/\tau_1\mathbb{Z}) \times \dots \times (\R/\tau_L\mathbb{Z})\,.
$$
Choose the space
$$
D= [\widetilde C^{1}([0,\pi])]^M \times [C^{1}([0,\pi])]^M \times [\widetilde C^1([0,\pi])]^L\times [C^1_0([0,\pi])]^L \,,
$$
and define a bilinear form $\BB : D \times D \to \R$ as in~\eqref{eq:of:bilinear}.

\medbreak
At this stage, we can prove similar results for the higher dimension as Proposition~\ref{pro:of:biliear:form}, Lemma~\ref{lem:of:L1:greater:than}, and Proposition~\ref{pro:of:critical:points} which will assure by Theorem~\ref{th:Szulkin} that problem~\eqref{sys:main:modified:2:higher}--\eqref{eq:of:npbc} has at least $M+L+1$ geometrically distinct solutions. We now need the following result to complete the proof in the higher dimension. The proof is very similar to the proof of Proposition~\ref{pro:final:low}, so we omit it for briefness.

\begin{pro}
    For $\lambda> \ov{C}$, where $\ov{C}$ is as in~\eqref{eq:of:bddness:of:nabla:h}, system~\eqref{sys:main:modified:2:higher} does not have any solution $z(t)=(q(t),p(t),u(t),v(t))$ departing with $p(0) \in \mathring{\mathcal{D}} $ satisfying the boundary conditions~\eqref{eq:of:npbc}.
\end{pro}

\section{Examples of applications and final remarks }\label{sec:4}

As an example, in the case $M=L=1$, consider the problem
\begin{equation}\label{system:for:the:example}
\begin{cases}
\ddot{q}+A \sin q  =e(t) +\partial_{q} P(t,q,u)\,,\\
\ddot{u} = h(u) + \partial_{u} P(t,q,u)\,,\\
q(a)=q(b), \quad p(a) = p(b),\\
            u'(a)=0=u'(b)\,.
\end{cases}
\end{equation}
where the constant $A$ is positive, the function $h:\R \to \R$ is continuous and $P:[a,b] \times \R^2 \to \R$ is continuously differentiable, $2 \pi$-periodic in $q$ and $\tau$-periodic in $u$. We assume that the function $h$ is also $\tau$-periodic with $\int_{0}^{\tau} h(s) ds =0$. Assuming
$$
\int_a^b e(t)\,dt=0\,,
$$
and setting $E(t)=\int_a^te(s)\,ds$, system in problem~\eqref{system:for:the:example} is equivalent to
\begin{equation*}
\begin{cases}
\dot{q}= p+E(t)\,, \quad
\dot{p} = -A \sin q  + \partial_{q} P(t,q,u)\,,\\
\dot{u} = v\,,\quad
\dot{v} = h(u) + \partial_{u} P(t,q,u)\,.
\end{cases}
\end{equation*}

\medbreak
It has been proved in~\cite[Section 3]{FonUll2024DIE}, that the twist condition $A2$ holds true. As an example for $h$ in the second equation of~\eqref{system:for:the:example}, we can take $h(u)=-\sin u$ or the {\em saw-tooth} function $h(u)= \arcsin{(\sin u)}$ (here $\tau=2 \pi$). We can see that all the assumptions of Theorem~\ref{thm:main} are satisfied, therefore there are  {\em at least three} geometrically distinct solutions of problem~\eqref{system:for:the:example}.

\medbreak
It is noted that the above example generalizes a classical theorem in~\cite{MawWil1984JDE} by Mawhin and Willem on the multiplicity of periodic solutions for the pendulum equation.

\medbreak

We now consider some variants of Theorem~\ref{thm:main}. 
To this aim, we first recall some definitions.
A closed convex bounded subset $\mathcal{D}$ of $\mathbb{R}^M$ having nonempty interior $\mathring {\mathcal{D}}$ is said to be a convex body of $\mathbb{R}^{M}$. If we assume that $\mathcal{D}$ has a smooth boundary, then we denote the unit outward normal at $\xi \in \partial \mathcal{D}$ by $\nu_{\mathcal{D}}(\xi)$. Moreover, we say that $\mathcal{D}$ is strongly convex if for any $r \in \partial \mathcal{D}$, the map $\mathcal{F}: \mathcal{D} \to\R$ defined by $\mathcal{F}(\xi)= \langle \xi-r , \nu_{\mathcal{D}}(r)\rangle$ has a unique maximum point at $\xi=r$.

\medbreak

We first state the following ``avoiding rays'' assumption studied in \cite{FonGarSfe2023JMAA, FonUll2024JDE, FonUll2024DIE} as follows.

\medbreak

\noindent $A2'$. There exists a convex body $\mathcal{D}$ of $\mathbb{R}^{M}$, having a smooth boundary, such that for $\sigma \in \{-1,1\}$ and for every $C^1$-function $\mathcal{W}:[a,b] \to \R^{2L}$, all the solutions $(q,p)$ of system~\eqref{eq:of:hamiltonian:unperterbed:for:higher}
starting with $p(a)\in{\cal D}$ are defined on $[a,b]$, and 
$$
p(a) \in \partial{\mathcal{D}}\quad\Rightarrow\quad  
q(b)-q(a) \notin \{ \sigma\lambda\, \nu_{\mathcal{D}}(p(a)):\lambda\ge0 \}\,. 
$$

\medbreak

Let us now state the following ``indefinite twist'' assumption which has been studied in \cite{FonGarSfe2023JMAA, FonUll2024JDE, FonUll2024DIE}.

\medbreak

\ni $A2''$. There are a strongly convex body ${\cal D}$ of $\R^M$ having a smooth boundary and a symmetric regular $M \times M$ matrix $\mathbb{A}$ such that for every $C^1$-function $\mathcal{W}:[a,b] \to \R^{2L}$, all the solutions $(q,p)$ of system~\eqref{eq:of:hamiltonian:unperterbed:for:higher}
starting with $p(a)\in{\cal D}$ are defined on $[a,b]$, and 
$$
p(a) \in \partial{\mathcal{D}}\quad\Rightarrow\quad    
\langle q(b)-q(a)\,,\, \mathbb{A} \nu_{\mathcal{D}}(p(a)) \rangle > 0\,.
$$

\begin{Thm}
If in the statement of Theorem~\ref{thm:main} we replace assumption $A2$ by $A2'$ or by $A2''$, the same conclusion holds.
\end{Thm}

\medbreak
Let us now mention some possible further developments and open problems in this direction.
\medbreak

\noindent
{\sl 1.} The coupling of a Hamiltonian system having twist condition with a system having lower and upper solutions under periodic boundary conditions was examined in \cite{FonGarSfe2023JMAA, FonUll2024JDE}. In contrast, coupled Hamiltonian systems with a system having lower and upper solutions under Neumann-type boundary conditions was studied in \cite{FonMamObeSfe2024NODEA, Mam2024TMNA}. It would be worth to explore mixed boundary conditions, such as those in \eqref{eq:of:npbc}, for this type of coupled Hamiltonian systems.

 \medbreak
 
\noindent
{\sl 2.}  In \cite{FonUll2024NODEA}, the authors examined the coupling of a Hamiltonian system with a positive, positively-$(p,q)$-homogeneous Hamiltonian system associated with a Neumann-type boundary conditions. They also discussed the periodic problem, assuming a {\em twist condition} for the first system. It would be intriguing to study a coupled Hamiltonian system where the first system with a {\em twist assumption} satisfies a periodic conditions, and the second system is a positively-$(p,q)$-homogeneous Hamiltonian system with a Neumann-type boundary conditions and some {\em nonresonance} condition.

 \medbreak

\noindent
{\sl 3.} Fonda and Gidoni~\cite{FonGid2020NODEA} extended the higher-dimensional Poincar\'e--Birkhoff Theorem for the periodic problem associated with Hamiltonian systems coupling twisting components with nonresonant linear components, where they used a symmetric matrix for the second system. On the other hand, the similar Hamiltonian system with Neumann-type boundary conditions was treated in \cite{FonUll2024NODEA} for a special type of symmetric matrix. The case of using general symmetric matrix in \cite{FonUll2024NODEA} is still open. An interesting open problem arises for such systems by considering periodic boundary conditions for the first system, while Neumann-type boundary conditions for the second one with general symmetric matrix.

\medbreak

\noindent
{\sl 4.} In \cite{FonSfeToa2024Pre}, the authors proved a multiplicity result for the periodic problem associated with a Hamiltonian system whose Hamiltonian function has a twisting part and a nonresonant part. They also analyzed a possible approach to resonance together with some kind of Landesman--Lazer conditions. One could try to extend their results by taking Neumann-type boundary conditions for the second system.

\medbreak

\noindent
{\sl 5.} Multiplicity results were obtained in \cite{MamUll2024Pre} for the periodic problem associated with a Hamiltonian system coupling a system having a Poincar\'e--Birkhoff twist-type structure with another planar system whose nonlinearity lies between the gradients of two positive and positively $2$-homogeneous Hamiltonain functions. Nonresonance, simple resonance and  double resonance situations were treated by imposing some kind of Landesman--Lazer conditions at both sides.
It would be interesting to know if parallel results can be proved in the setting of this paper.

\bigbreak

\ni {\bf Acknowledgement}. The research contained in this paper was carried out within the framework of DEG1 - Differential Equations Group Of North-East of Italy. I am deeply grateful to Professor Alessandro Fonda and Professor Andrea Sfecci for generously allowing me to explore this research problem. Their insightful guidance  and valuable discussions significantly contributed to improving the results of this manuscript.



\vspace{0.5cm}
\footnotesize
\noindent Author's address:

\bigbreak

\begin{tabular}{l}
Wahid Ullah\\
Dipartimento di Matematica, Informatica e Geoscienze\\
Universit\`a degli Studi di Trieste\\
P.le Europa 1, 34127 Trieste, Italy\\
e-mail: wahid.ullah@phd.units.it
\end{tabular}

\bigbreak

\noindent Mathematics Subject Classification: 34C25.

\medbreak

\noindent Keywords: Hamiltonian systems; periodic conditions; Neumann boundary conditions;  Poincar\'e--Birkhoff Theorem.

\end{document}